\pdfoutput=1
\documentclass[psamsfonts, reqno, 12pt]{amsart}
\usepackage{amsfonts,amsmath, amsthm, amssymb, latexsym, epsfig}
\usepackage[all]{xy}
\usepackage{xspace}
\usepackage{enumerate}
\usepackage{hyperref}
\usepackage[mathscr]{eucal}

        \advance \topmargin by -\headheight
        \advance \topmargin by -\headsep

        \evensidemargin \oddsidemargin
\RequirePackage{amsfonts}
\RequirePackage{amsmath}
\RequirePackage{amsthm}
\RequirePackage{amssymb}
\RequirePackage{latexsym}
\RequirePackage{epsfig}
\RequirePackage{verbatim}
\RequirePackage[mathscr]{eucal}
\RequirePackage[all]{xy}

\newenvironment{smallpmatrix}{%
  \left(\begin{smallmatrix}}{\end{smallmatrix}\right)}

        \newcommand{\cstar}{\mbox{$C^*$}}
        \newcommand{\ror}{\mathcal R}

        \newcommand{\ftn}[3]{#1 : #2 \longrightarrow #3}
        \newcommand{\setof}[2]{\ensuremath{\left\{ #1 \: : \: #2 \right\}}}

        \newcommand{\ksix}{\ensuremath{K_{\mathrm{six}}}\xspace}
        
        \newcommand{\mc}[1]{\mathcal{#1}}
        
        \newcommand{\msf}[1]{\mathsf{#1}}
        \newcommand{\mrm}[1]{\mathrm{#1}}

        \newcommand{\mbb}[1]{\mathbb{#1}}
        \newcommand{\mfk}[1]{\mathfrak{#1}}
        
        \newcommand{\Z}{\ensuremath{\mathbb{Z}}\xspace}
        \newcommand{\C}{\ensuremath{\mathbb{C}}\xspace}

        \newcommand{\kk}{\ensuremath{\mathit{KK}}\xspace}

        \newcommand{\kkE}{\ensuremath{\mathit{KK}_{\mathcal{E}}}\xspace}
        \newcommand{\tksix}{\ensuremath{\underline{\mathit{K}}_{\mathrm{six}}}\xspace}
        \newcommand{\tk}{\ensuremath{\underline{\mathit{K}}}\xspace}

        \newcommand{\id}{\ensuremath{\operatorname{id}}}
        
        \newcommand{\diagA}[8]
        {\xymatrix
        {       & \ar[d]                                & \ar[d]                        & \ar[d]                                &       \\
        \ar[r] &        #1 \ar[r] \ar[d]                & #2 \ar[r] \ar[d]              & #3 \ar@{=}[d] \ar[r]          &       \\
        \ar[r]  &       #4 \ar[r] \ar[d]                & #5 \ar[r] \ar[d]              & #6 \ar[d] \ar[r]                      &       \\
        \ar[r]  &       #7 \ar@{=}[r] \ar[d]    & #8 \ar[r] \ar[d]      & 0 \ar[r] \ar[d]                       &       \\
                &                                       &                               &                                       & 
        }
        }
        
        \newcommand{\diagB}[8]
        {\xymatrix
        {       & \ar[d]                                        & \ar[d]                        & \ar[d]                                        &       \\
        \ar[r] &        #1 \ar@{=}[r] \ar[d]            & #2 \ar[r] \ar[d]              & 0 \ar[d] \ar[r]                               &       \\
        \ar[r]  &       #3 \ar[r] \ar[d]                        & #4 \ar[r] \ar[d]              & #5 \ar@{=}[d] \ar[r]                  &       \\
        \ar[r]  &       #6 \ar[r] \ar[d]                        & #7 \ar[r] \ar[d]      & #8 \ar[r] \ar[d]                              &       \\
                &                                               &                               &                                               & 
        }
        }
        
        \newcommand{\diagC}[8]
        {\xymatrix
        {       & \ar[d]                                        & \ar[d]                                & \ar[d]                                &       \\
        \ar[r] &        #1 \ar[r] \ar@{=}[d]            & #2 \ar[r] \ar[d]                      & #3 \ar[d] \ar[r]                      &       \\
        \ar[r]  &       #4 \ar[r] \ar[d]                        & #5 \ar[r] \ar[d]                      & #6 \ar[d] \ar[r]                      &       \\
        \ar[r]  &       0 \ar[r] \ar[d]                 & #7 \ar@{=}[r] \ar[d]  & #8 \ar[r] \ar[d]                      &       \\
                &                                               &                                       &                                       & 
        }
        }

\newcommand{\ourdd}[2]{\mbb{I}_{#1}^{#2}}
        
\theoremstyle{plain}
        \newtheorem{thm}{Theorem}
        \newtheorem{lemma}[thm]{Lemma}
        \newtheorem{theorem}[thm]{Theorem}
        \newtheorem{proposition}[thm]{Proposition}

        \theoremstyle{definition}
        \newtheorem{definition}[thm]{Definition}
        
        \newtheorem{remark}[thm]{Remark}
        
        \newtheorem{example}[thm]{Example}

        \numberwithin{equation}{section}
        \numberwithin{figure}{section}

\begin{document}
        \title{Nonsplitting in Kirchberg's ideal-related $\kk$-theory}
        \author{S{\o}ren Eilers}
        \author{Gunnar Restorff}
        \address{Department of Mathematical Sciences \\
        University of Copenhagen\\
        Universitetsparken~5 \\
        DK-2100 Copenhagen, Denmark}
        \email{eilers@math.ku.dk}       
        \email{restorff@math.ku.dk } 
       \author{Efren Ruiz}
        \address{Department of Mathematics \\
        University of Hawaii Hilo \\
        200 W. Kawili St. \\
        Hilo, Hawaii 96720 USA}
        \email{ruize@hawaii.edu}
        \date{\today}
        

        \keywords{KK-theory, UCT}
        \subjclass[2000]{Primary: 46L35}

        \begin{abstract}
A universal coefficient theorem in the setting of Kirchberg's
ideal-related $\kk$-theory was obtained in the fundamental case of a
\cstar-algebra with one
specified ideal by Bonkat in
\cite{ab:bkkpsc} and proved there to split, unnaturally, under certain
conditions. Employing certain $K$-theoretical information derivable
from the given operator algebras in a way introduced here, we shall
demonstrate that Bonkat's UCT does not split in general. Related
methods lead to information on the complexity  of the $K$-theory which
must be used to
classify $*$-isomorphisms for purely infinite $\cstar$-algebras with
one non-trivial ideal.
        \end{abstract}
        \maketitle

%

\section{Introduction}\label{sec-intro}

The $\kk$-theory introduced by Kasparov 
(\cite{ggk:okec}) is one of the most important
tools in the theory of classification of \cstar-algebras, of use especially for simple \cstar-algebras. 
Recently, Kirchberg has developed the socalled ideal-related
$\kk$-theory --- a generalisation of Kasparov's $\kk$-theory
which takes into account the ideal structure of the algebras
considered ---
and obtained strong isomorphism theorems for stable, nuclear,
separable, strongly purely infinite \cstar-algebras
(\cite{ek:nkmkna}). 
The results obtained by Kirchberg establish 
ideal-related $\kk$-theory as an  essential
tool in the classification theory of non-simple \cstar-algebras. 

$\kk$-theory is a bivariant functor;
to obtain a real classification result
one needs a univariant classification functor instead. 
For ordinary $\kk$-theory this is obtained (within the bootstrap category)
by invoking the Universal Coefficient Theorem (UCT) of Rosenberg and
Schochet:

\begin{theorem}[{Rosenberg-Schochet's UCT, \cite{jrcs:ktuctkgk}}] \label{thm:rsUCT}
  Let $A$ and $B$ be separable \cstar-algebras in the bootstrap category $\mc{N}$. 
  Then there is a short exact sequence
  \begin{equation*}
    \mrm{Ext}_{ \Z }^{1} ( K_*(A),K_*(SB) ) \hookrightarrow \kk ( A,B ) \overset{\gamma}{\twoheadrightarrow} \mrm{Hom}_{ \Z } ( K_*(A),K_*(B) )
  \end{equation*}
  (here $K_*(-)$
  denotes the graded group $K_0(-)\oplus K_1(-)$).
  The sequence is natural in both $A$ and $B$, and splits  (unnaturally, in general). 
  Moreover, an element $x$ in $\kk(A,B)$ is
  invertible if and only if $\gamma(x)$ is an isomorphism.
\end{theorem}

This UCT allows us to turn isomorphism results (such as Kirchberg-Phillips'
theorem \cite{ekncp:eeccao})
into strong classification theorems. 
Moreover, using the splitting, it allows us to determine completely
the additive structure of the \kk-groups. 

To transform Kirchberg's general result into a strong classification
theorem, one would need a UCT for ideal-related $\kk$-theory. 
This was achieved by Bonkat (\cite{ab:bkkpsc}) in the special case
where the specified ideal structure is just a single ideal. Progress
into more general cases  with finitely many ideals has recently been announced by Mayer-Nest and by the second named
author, but in this paper we will
only consider the case with one specified ideal:

\begin{theorem}%
  [{Bonkat's UCT, \cite[Satz 7.5.3, Satz 7.7.1, and Proposition 7.7.2]{ab:bkkpsc}}]\label{thm:bonkatuct}
  Let $e_{1}$ and $e_{2}$ be extensions of separable, nuclear
  \cstar-algebras in the bootstrap category $\mc{N}$. 
  Then there is a short exact sequence
  \begin{equation*}
    \mrm{Ext}_{ \mc{Z}_{6} }^{1} ( \ksix ( e_{1} ) , \ksix ( S e_{2} ) )
    \hookrightarrow \kkE ( e_{1} , e_{2} ) \overset{\Gamma}{\twoheadrightarrow} \mrm{Hom}_{ \mc{Z}_{6} } ( \ksix ( e_{1} ) , \ksix ( e_{2} ) )
  \end{equation*}
  (here $\ksix(-)$ is the standard cyclic six term exact sequence,
  $\mc{Z}_6$ is the category of cyclic six term chain complexes, and $Se$
  denotes the extension obtained by tensoring all the \cstar-algebras 
  in the extension $e$ with $C_{0} ( 0 , 1 )$). 
  The sequence is natural in both $e_1$ and $e_2$. 
  Moreover, an element $x\in\kkE(e_1,e_2)$ is invertible if and only if
  $\Gamma(x)$ is an isomorphism.
\end{theorem}

Bonkat leaves open the question of whether this UCT splits in
general. We prove here that this is not always the case, even in the
fundamental case considered by Bonkat (see
Proposition~\ref{prop:nonsplitting}(\ref{prop:nonsplitting-item1})
below).

This observation tells us --- in contrast to the ordinary $\kk$-theory
--- that we cannot, in general, completely determine the additive
structure of $\kkE$ just by using the UCT.  It is comforting to note,
as may be inferred from the results in \cite{mr:ceccstesk},
\cite{segr:rccconi} and \cite{grer:rccconiII}, that this has
only marginal impact on the usefulness of Bonkat's result in the
context of classification of e.g.\ the \cstar-algebras considered by
Kirchberg. But as we shall see it has several repercussions concerning
the classification of homomorphisms and automorphisms of such
$\cstar$-algebras, and opens an intriguing discussion --- which it is our ambition
to close elsewhere (\cite{segrer:ancka}) in the important special
case of Cuntz-Krieger algebras satisfying condition (II) --- on the
nature of an invariant classifying such morphisms. 

Indeed, examples abound in classification theory in which the
invariant needed to classify automorphims up to approximate unitary
equivalence on a certain class of
\cstar-algebras is more complicated than the classifying invariant for
the algebras themselves. For instance, even though $K_*(-)$ is a
classifying invariant for stable Kirchberg algebras (i.e.\  nuclear, separable, simple, purely infinite
  \cstar-algebras) one needs to turn to \emph{total} $K$-theory ---
the collection of $K_*(-)$ and all torsion coefficient groups
$K_*(-;\Z_n)$ --- in
order to obtain exactness of 
  \begin{equation}\label{eq-kirch}
    \{1\}\rightarrow
    \overline{\mrm{Inn}}(A)
    \rightarrow\mrm{Aut}(A) \rightarrow\mrm{Aut}_{\Lambda}(\underline{K}(A))
    \rightarrow\{1\},
  \end{equation}
  where $\overline{\mrm{Inn}}(A)$ is the group of
  automorphisms of $A$ that are approximately unitarily equivalent to
  $\id_A$ and the subscript $\Lambda$ indicates that the group
isomorphism on $\underline{K}(A)$ must commute with all the natural
\emph{Bockstein} operations.

The appearance of total $K$-theory in \eqref{eq-kirch} is explained by
the  Universal Multicoefficient Theorem obtained by 
Dadarlat and Loring in \cite{mdtal:umctkg}:

\begin{theorem}[{Dadarlat-Loring's UMCT, \cite{mdtal:umctkg}}]\label{umct}
  Let $A$ and $B$ be separable \cstar-algebras in the bootstrap
  category $\mc{N}$.
  Then there is a short exact sequence 
  \begin{equation*}
    \mrm{Pext}_{ \Z }^{1} ( K_*(A),K_*(SB) ) \hookrightarrow \kk (A,B) 
    \twoheadrightarrow\mrm{Hom}_\Lambda(\tk(A),\tk(B) )
  \end{equation*}
  (here $\mrm{Pext_\Z^1}$ denotes the subgroup of $\mrm{Ext}_\Z^1$
  consisting of pure extensions, and $\mrm{Hom}_\Lambda$ denotes the group of
  homomorphisms respecting the Bockstein operations). 
  The sequence is natural in both $A$ and $B$, and an element $x$ in $\kk(A,B)$ is
  invertible if and only if the induced element is an isomorphism.
  Moreover, $\mrm{Pext}_\Z^1(K_*(A),K_*(SB))$ is zero whenever the
  $K$-theory of $A$ is finitely generated.
\end{theorem}

Dadarlat has pointed out to us that although \cite{mdtal:umctkg} states that
the UMCT splits  in general, this is not true. The problem
can be traced to one in  \cite{cls:umskg}, cf.\
\cite{cls:cumskg}
and \cite{cls:fskgII}. 

In the stably finite case, as exemplified by stable real rank zero $AD$
algebras, the UMCT leads to exactness of
  \begin{equation}\label{eq-dl}
    \{1\}\rightarrow
    \overline{\mrm{Inn}}(A)
    \rightarrow\mrm{Aut}(A) \rightarrow\mrm{Aut}_{\Lambda,+}(\underline{K}(A))
    \rightarrow\{1\},
  \end{equation}
in which the subscript  ``$+$'' indicates the presence of positivity
conditions (see \cite{mdtal:umctkg} for details). 
Noting the way the usage of a six term exact sequence in
\cite{mr:ceccstesk} parallels the
usage of positivity in the stably finite case (cf.\
\cite{mdse:ccpikc}) it is natural to
speculate (as indeed the first  named author did at The First Abel
Symposium, cf.\ \cite{segr:rccconi}) that by combining all coefficient six term exact sequences
into an invariant $\tksix(-)$ one obtains an exact sequence of the form
\begin{equation}\label{autoeq1}
\xymatrix{
\{ 1 \} \ar[r] & \overline{\mrm{Inn}} ( e ) \ar[r] & \mrm{Aut}(e) \ar[r] & \mrm{Aut}_\Lambda ( \tksix ( e ) ) \ar[r] & \{1\},
}
\end{equation}
and to search for a corresponding UMCT along the lines of
Theorem \ref{umct}.

This sequence is clearly a chain complex, but as we will see, the natural map from $\kkE(e_1,e_2)$ to
$\mrm{Hom}_\Lambda(\tksix(e_1),\tksix(e_2))$ is not injective nor
is it surjective in general for extensions $e_1$ and $e_2$ with
finitely generated $K$-theory (see
Proposition~\ref{prop:nonsplitting}(\ref{prop:nonsplitting-item2}),(\ref{prop:nonsplitting-item3})),
and 
we will  give an example of an extension of stable Kirchberg
algebras in the bootstrap 
category $\mc{N}$ with finitely generated $K$-theory, such that
\eqref{autoeq1} is only exact at $\overline{\mrm{Inn}} ( e )$, telling
us in unmistakable terms that this is the wrong invariant to use. 

Our methods are based on computations related to a
class of extensions which, we believe, should be thought of as a
substitute for the total $K$-theory of relevance in the classification
of, e.g., non-simple, stably finite \cstar-algebras with real rank
zero. We shall undertake a more systematic study of these objects
elsewhere, and show there how they may be employed to the task of computing 
Kirchberg's groups $\kkE(-,-)$.

\section{Preliminaries}\label{sec-prelim}

We first set up some notation that will be used throughout.

\begin{definition}\label{def:kkthycoeff}
Let $n\geq2$ be an integer and
denote the non-unital dimension drop algebra by  $\ourdd{n}{0} = \setof{ f \in C_{0} ( (0,1] , \msf{M}_{n} )  }{
f(1) \in \C 1_{ \msf{M}_{n} } }$.  Then $\ourdd{n}{0}$ fits into the
short exact sequence 
\begin{equation*}
\mfk{e}_{n,0} :  S\msf{M}_{n} \hookrightarrow \ourdd{n}{0} \twoheadrightarrow \C.
\end{equation*}
It is well known that $K_{0} ( \ourdd{n}{0} ) = 0$ and $K_{1} ( \ourdd{n}{0} ) = \Z_{n}$, where $\Z_{n}$ denotes the cyclic abelian group with $n$ elements.  

Let $\mfk{e}_{n , 1} :  S \mbb{C} \hookrightarrow \ourdd{n}{1} \twoheadrightarrow \ourdd{n}{0}$ be the extension obtained from the mapping cone of the map $\ourdd{n}{0} \twoheadrightarrow \C$.  The diagram
\begin{equation}
\vcenter{
\xymatrix{
 0 \ar[r] \ar[d] & S\C \ar@{^{(}->}[d] \ar@{=}[r] & S\C \ar@{^{(}->}[d] \\
 S\msf{M}_{n} \ar@{=}[d] \ar@{^{(}->}[r] & \ourdd{n}{1} \ar@{>>}[d] \ar@{>>}[r] & C\C  \ar@{>>}[d] \\
 S\msf{M}_{n} \ar@{^{(}->}[r] & \ourdd{n}{0} \ar@{>>}[r] & \C  
}
}
\end{equation}  
is commutative and the columns and rows are short exact sequences.  Note that the $*$-homomorphism from $S \msf{M}_{n}$ to $\ourdd{n}{1}$ induces a $\kk$-equivalence.  

Let $\mfk{e}_{n , 2} :  S \ourdd{n}{0} \hookrightarrow \ourdd{n}{2}
\twoheadrightarrow \ourdd{n}{1}$ be the extension obtained from the
mapping cone of the canonical map $\ourdd{n}{1} \twoheadrightarrow \ourdd{n}{0}$.  Then the diagram
\begin{equation}
\vcenter{
\xymatrix{
 0 \ar[r] \ar[d] & S\ourdd{n}{0} \ar@{^{(}->}[d] \ar@{=}[r] & S\ourdd{n}{0} \ar@{^{(}->}[d] \\
 S\C \ar@{=}[d] \ar@{^{(}->}[r] & \ourdd{n}{2} \ar@{>>}[d] \ar@{>>}[r] & C\ourdd{n}{0}  \ar@{>>}[d] \\
 S\C \ar@{^{(}->}[r] & \ourdd{n}{1} \ar@{>>}[r] & \ourdd{n}{0}  
}
}
\end{equation}  
is commutative and the columns and rows are short exact sequences.  Note
that the $*$-homomorphism from $S\C$ to $\ourdd{n}{2}$ induces a
$\kk$-equivalence. This implies, with a little more work, that we get no new
$K$-theoretical information from considering objects
$\ourdd{n}{k}$ or  $\mfk{e}_{n , k}$ for $k>2$. Note also that the
\cstar-algebras 
$\ourdd{n}{0},\ourdd{n}{1}$ and $\ourdd{n}{2}$ are NCCW complexes of
dimension 1, 1, and 2, respectively, in the sense of
\cite{elp:sar}. See Figure \ref{nccws}.

\begin{figure}
\begin{center}
\begin{tabular}{ccc}
$\ourdd{n}0$&$\ourdd{n}1$&$\ourdd{n}2$\\
\qquad\includegraphics[scale=0.39,angle=270]{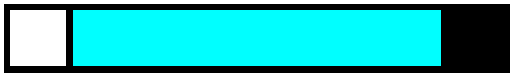}\qquad&\qquad\includegraphics[scale=0.39,angle=270]{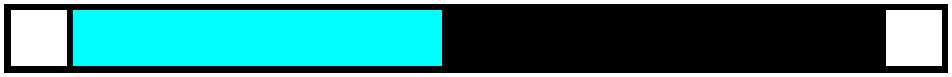}\qquad&\qquad\includegraphics[scale=0.39,angle=270]{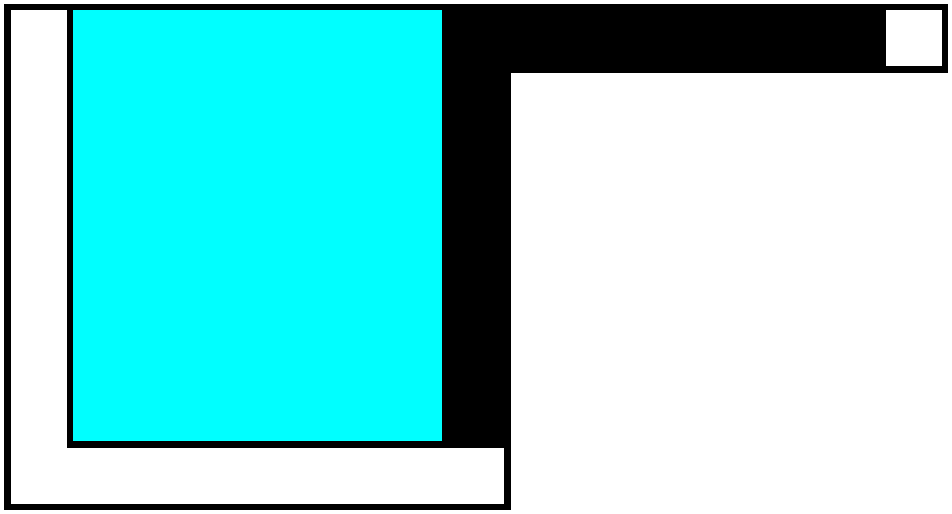}\qquad\end{tabular}
\end{center}
\bigskip
\mbox{[Fibre legends:
\raisebox{-0.05cm}{\includegraphics[width=0.4cm]{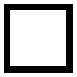}}=$0$,
\raisebox{-0.05cm}{\includegraphics[width=0.4cm]{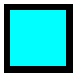}}=$\msf{M}_{n}$,
\raisebox{-0.05cm}{\includegraphics[width=0.4cm]{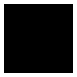}}=$\mathbb C$]}
\caption{NCCW structure of $\ourdd ni$ }
\label{nccws}
\end{figure}

Let $e :  A_{0} \hookrightarrow A_{1} \twoheadrightarrow A_{2}$ be an
extension of \cstar-algebras.  We have an ``ideal-related $K$-theory
with $\Z_{n}$-coefficients'' denoted by $\ksix ( e ; \Z_{n} )$.  More
precisely, $\ksix ( e ; \Z_{n} )$ denotes the six term exact sequence 
\begin{equation*}
\xymatrix{
K_{0} ( A_{0} ; \Z_{n} ) \ar[r] & K_{0} ( A_{1} ; \Z_{n} ) \ar[r] & K_{0} ( A_{2} ; \Z_{n} ) \ar[d] \\
K_{1} ( A_{2} ; \Z_{n} ) \ar[u] & K_{1} ( A_{1} ; \Z_{n} ) \ar[l] & K_{1} ( A_{0} ; \Z_{n} ) \ar[l] 
}
\end{equation*}  
obtained by applying the covariant functor $\kk^{*} ( \ourdd{n}{0} , -
)$ to the extension $e$. 

Let us denote the standard six term exact sequence in $K$-theory by
$\ksix(e)$. The collection consisting of $\ksix(e)$ and
$\ksix ( e ; \Z_{n} )$ for all $n\geq 2$ will be denoted by $\tksix ( e
)$.  A homomorphism from $\tksix ( e_{1} )$ to $\tksix ( e_{2} )$
consists of a morphism from $\ksix(e_1)$ to $\ksix(e_2)$ along with an infinite family of morphisms from $\ksix ( e_{1} ; \Z_{n}
)$ to $\ksix ( e_{2} ; \Z_{n} )$ respecting the Bockstein operations
in $\Lambda$.
We will denote the group of homomorphisms from $\tksix ( e_{1} )$ to
$\tksix ( e_{2} )$ by $\mrm{Hom}_{ \Lambda } ( \tksix ( e_{1} ) ,
\tksix ( e_{2} ) )$.   
We turn $\tksix$ into a functor in the obvious way.  

\begin{lemma} There is  a natural homomorphism 
\begin{equation*}
\Gamma_{e_{1}, e_{2}} : \kkE ( e_{1} , e_{2}  ) \longrightarrow \mrm{Hom}_{ \Lambda } ( \tksix ( e_{1} ) , \tksix ( e_{2} ) ).
\end{equation*}
\end{lemma}

\begin{proof}
A computation shows that $\ksix ( - ; \Z_{n} )$ is a stable, homotopy invariant, split exact functor since $\kk$ satisfies these properties.  Therefore, $\tksix ( - )$ is a stable, homotopy invariant, split exact functor.  Hence, for every fixed extension $e_{1}$ of \cstar-algebras, $\mrm{Hom}_{ \Lambda } ( \tksix ( e_{1} ) , \tksix ( - ) )$ is a stable, homotopy invariant, split exact functor.  By Satz 3.5.9  of \cite{ab:bkkpsc}, we have a natural transformation $\Gamma_{e_{1} , -}$ from $\kkE( e_{1} , - )$ to $\mrm{Hom}_{ \Lambda } ( \tksix ( e_{1} ) , \tksix ( - ) )$ such that $\Gamma_{e_{1} , - }$ sends $[ \mrm{id}_{e_{1}} ]$ to $\tksix ( \id_{e_{1}} )$.  Arguing as in the proof of Lemma 3.2 of \cite{nh:ck}, we have that 
\begin{equation*}
\Gamma_{e_{1} , e_{2} } : \kkE ( e_{1} , e_{2} ) \longrightarrow \mrm{Hom}_{ \Lambda } ( \tksix ( e_{1} ) , \tksix ( e_{2} ) )
\end{equation*} 
is a group homomorphism.
\end{proof}

Another collection of groups that we will use in this  paper is the
following:  for each $n\geq 2$, set 
\begin{equation*}
K_{ \mc{E} } ( e ; \Z_{n} ) = \bigoplus_{ i = 0}^{ 2} \left( \kkE^{*} ( \mfk{e}_{ n , i } , e ) \oplus \kk^{*} ( \ourdd{n}{0} , A_{i} ) \oplus \kk^{*} ( \C , A_{i} ) \right).
\end{equation*}

\section{Examples}\label{sec-examples}

Accompanied with the groups $\kkE^{*} ( \mfk{e}_{ n , i} , e )$ are naturally defined diagrams, which will be systematically described in a forthcoming paper.  For now, we will use these groups to show the following:
\end{definition}

\begin{proposition}\label{prop:nonsplitting}
\begin{enumerate}[(1)]
\item\label{prop:nonsplitting-item1}
  The UCT of Bonkat (Theorem~\ref{thm:bonkatuct}) 
  does not split in general.
\item\label{prop:nonsplitting-item2}
  There exist $e_{1}$ and $e_{2}$ extensions of separable,
  nuclear \cstar-algebras in the bootstrap category $\mc{N}$ of
  Rosenberg and Schochet \cite{jrcs:ktuctkgk} such that the six term
  exact sequence of $K$-groups associated to $e_{1}$ is finitely
  generated and  
  \begin{equation*}
    \Gamma_{e_{1} , e_{2} } : \kkE( e_{1} , e_{2} ) 
    \longrightarrow \mrm{Hom}_{ \Lambda } ( \tksix ( e_{1} ) , \tksix ( e_{2} ) )
  \end{equation*}
  is not injective.
\item\label{prop:nonsplitting-item3}
  There exist $e_{1}$ and $e_{2}$ extensions of separable,
  nuclear \cstar-algebras in the bootstrap category $\mc{N}$ of
  Rosenberg and Schochet \cite{jrcs:ktuctkgk} such that the six term exact sequence of $K$-groups associated to $e_{1}$ is finitely generated and
  \begin{equation*}
    \Gamma_{e_{1} , e_{2}} : \kkE( e_{1} , e_{2} ) \longrightarrow \mrm{Hom}_{ \Lambda } ( \tksix ( e_{1} ) , \tksix ( e_{2} ) )
  \end{equation*}
  is not surjective.
\end{enumerate}
\end{proposition}


The proposition will be proved through a series of examples. The
following example shows that the UCT of Bonkat does not split in general.   
Also it shows that there exist extensions $e_{1}$ and $e_{2}$ of
separable, nuclear \cstar-algebras in $\mc{N}$ with finitely generated
$K$-theory, such that $\Gamma_{ e_{1} , e_{2} }$ is not injective. 

\begin{example}\label{ex:nonsplitting}
Let $n$ be a prime number.  By Korollar 7.1.6 of \cite{ab:bkkpsc}, we have that 
\begin{equation*}
\xymatrix{
\Z \ar[r] & \Z \ar[r] & \kkE^{1} ( \mfk{e}_{n,0} , \mfk{e}_{n,1} ) \ar[r] & 0
}
\end{equation*}
is an exact sequence.  Therefore, $\kkE^{1} ( \mfk{e}_{ n , 0} , \mfk{e}_{n,1} )$ is a cyclic group.  By Korollar 7.1.6 of \cite{ab:bkkpsc}, $\kkE^{1} ( \mfk{e}_{n,0} , \mfk{e}_{n,1} )$ fits into the following exact sequence
\begin{equation*}
\xymatrix{
0 \ar[r]& \Z_{n} \ar[r]& \kkE^{1} ( \mfk{e}_{n,0} , \mfk{e}_{n,1} ) \ar[r]& \Z_{n} \ar[r]& 0.}
\end{equation*}
So, $\kkE^{1} ( \mfk{e}_{ n, 0 }  , \mfk{e}_{n,1} )$ is isomorphic to $\Z_{n^2}$.

An easy computation shows that $\mrm{Hom} ( \ksix ( \mfk{e}_{ n , 0} )
, \ksix ( S \mfk{e}_{n,1} ) )$ is isomorphic to $\Z_{n}$.  Using this
fact and the fact that $\kkE(\mfk{e}_{ n, 0 } , S\mfk{e}_{n,1} )\cong\kkE^{1} ( \mfk{e}_{ n, 0 } , \mfk{e}_{n,1} )$
is $\Z_{n^2}$, we immediately see that the UCT of Bonkat does not split in this case.

We would like to also point out another consequence of this example.
Since $n$ is prime and $\mrm{Ext}_{ \mc{Z}_{6} }^{1} ( \ksix (
\mfk{e}_{n,0} ) , \ksix ( \mfk{e}_{n,1} ) )$ injects into a proper
subgroup of $\kkE^{1} ( \mfk{e}_{n,0} , \mfk{e}_{n,1} )$, we have that
$\mrm{Ext}_{ \mc{Z}_{6} }^{1} ( \ksix ( \mfk{e}_{n,0} ) , \ksix (
\mfk{e}_{n,1} ) )$ is isomorphic to $\Z_{n}$.  Therefore, $n$
annihilates all $K$-theory information but $n$ does not annihilate $\kkE^{1} ( \mfk{e}_{n,0} , \mfk{e}_{n,1} )$. 

We will now show that the natural map $\Gamma_{ \mfk{e}_{ n , 0 } , S \mfk{e}_{n,1} }$
from $\kk_{ \mc{E} } ( \mfk{e}_{ n, 0 } , S \mfk{e}_{n,1} )$ to
$\mrm{Hom}_{ \Lambda } ( \tksix ( \mfk{e}_{n,0} ) , \tksix ( S
\mfk{e}_{n,1} ) )$ is not injective.  
Let $A_0\hookrightarrow A_1\twoheadrightarrow A_2$ and 
$B_0\hookrightarrow B_1\twoheadrightarrow B_2$ denote the extensions
$\mfk{e}_{n,0}$ and $S\mfk{e}_{n,1}$, respectively. 
Note that the corresponding six term exact sequences are (isomorphic to)
$$\vcenter{\xymatrix{
    0\ar[r]&0\ar[r]&\Z\ar[d]\\
    0\ar[u]&\Z_n\ar[l]&\Z\ar[l]
    }}
\quad\text{and}\quad
\vcenter{\xymatrix{\Z\ar[r]&\Z\ar[r]&\Z_n\ar[d]\\
    0\ar[u]&0\ar[l]&0\ar[l]}},$$
respectively. Using the UCT of Rosenberg and Schochet, a short computation shows that 
$n\prod_{i=0}^2\kk(A_i,B_i)=0$. 
Since all the $K$-theory is finitely generated, we have by
Dadarlat and Loring's UMCT that
$\mrm{Hom}_\Lambda(\tksix(\mfk{e}_{n,0}),\tksix(S\mfk{e}_{n,1}))$ is
isomorphic to a subgroup of $\prod_{i=0}^2\kk(A_i,B_i)$. 
Since the latter group has no element of order $n^2$ 
and 
$\kk_{ \mc{E} } ( \mfk{e}_{n,0} , S \mfk{e}_{n,1} )$ is isomorphic to $\Z_{n^{2}}$, we have that $\Gamma_{ \mfk{e}_{n,0} , S \mfk{e}_{n,1} }$ is not injective.
\end{example}


The above example also provides a counterexample to Satz 7.7.6 of \cite{ab:bkkpsc}.  The arguments in the proof of Satz 7.7.6 are correct but it appears that Bonkat overlooked the case were the six term exact sequences are of the form:
\begin{equation*}
\xymatrix{
0 \ar[r] & 0 \ar[r] & \ast \ar[d] \\
0 \ar[u] &  \ast \ar[l] & \ast \ar[l]
}\quad \quad
\xymatrix{
0 \ar[r] & 0 \ar[r] & 0 \ar[d] \\
\ast \ar[u] & \ast \ar[l] & \ast \ar[l]
}
\end{equation*}


Our next example shows that there exist extensions $e_{1}$ and
$e_{2}$ of separable, nuclear \cstar-algebras in $\mc{N}$ with
finitely generated $K$-groups, such that $\Gamma_{ e_{1} , e_{2} }$ is not surjective.

\begin{example}\label{ex:nonsurjtksix}
Let $n$ be a prime number.  Consider the following short exact sequences of extensions:
\begin{equation}\label{ex1}
\vcenter{
\xymatrix{
S \C \ar@{=}[r] \ar@{=}[d]& S \C \ar[r] \ar@{^{(}->}[d] & 0 \ar[d] \\
S \C \ar@{^{(}->}[r] \ar[d] & \ourdd{n}{1} \ar@{>>}[d] \ar@{>>}[r] & \ourdd{n}{0} \ar@{=}[d] \\
0 \ar[r]  & \ourdd{n}{0} \ar@{=}[r] & \ourdd{n}{0}  
}
}
\end{equation}
and
\begin{equation}\label{ex2}
\vcenter{
\xymatrix{
S\msf{M}_{n} \ar@{=}[r] \ar@{=}[d] & S\msf{M}_{n} \ar[r] \ar@{^{(}->}[d] & 0  \ar[d] \\
S\msf{M}_{n} \ar@{^{(}->}[r] \ar[d] & \ourdd{n}{0} \ar@{>>}[d] \ar@{>>}[r] & \C \ar@{=}[d] \\
0 \ar[r] & \C \ar@{=}[r] & \C
}
}
\end{equation}

By applying the bivariant functor $\kkE^{*}(-,-)$ to the above exact sequences of extensions with (\ref{ex1}) in the first variable and (\ref{ex2}) in the second variable and by Lemma 7.1.5 of \cite{ab:bkkpsc}, we get that the diagram 
\begin{equation*}
\xymatrix{
         & 0 \ar[d] & 0 \ar[d] & 0 \ar[d] & \\
 0 \ar[r] & 0 \ar[r] \ar[d] & \Z_{n} \ar[r] \ar[d] & \Z_{n} \ar@{=}[d] \ar[r] & 0 \\
 0 \ar[r] & \Z \ar[r] \ar[d]^{n} & \kkE ( \mfk{e}_{n,1} , \mfk{e}_{n,0} ) \ar[d] \ar[r] & \Z_{n} \ar[r] \ar[d] & 0 \\
 0 \ar[r]               & \Z \ar@{=}[r] \ar[d] & \Z \ar[d] \ar[r] &  0 \ar[r]  \ar[d] & 0\\
 0 \ar[r]       & \Z_{n} \ar[r] \ar[d]  & \Z_{n} \ar[r] \ar[d] & 0 \ar[r] \ar[d] & 0 \\
                        & 0                     & 0                             & 0                     &
                        }
\end{equation*}
is commutative.  By Korollar 3.4.6 of \cite{ab:bkkpsc} the columns and rows of the above diagram are exact sequences.  Therefore, we have that $\kkE ( \mfk{e}_{ n , 1} , \mfk{e}_{n,0} )$ is isomorphic to $\Z \oplus \Z_{n}$.  

A straightforward computation gives that $\ksix(\mfk{e}_{ n , 0})$ and
$\ksix(\mfk{e}_{ n , 0};\Z_m)$ are given by
\[
\xymatrix{
0\ar[r]&0\ar[r]&\Z\ar[d]&&0\ar[r]&\Z_{(m,n)}\ar[r]&\Z_{n}\ar[d]\\
0\ar[u]&\Z_n\ar[l]&\Z\ar[l]&&0\ar[u]&\Z_{n/(m,n)}\ar[l]&\Z_n\ar[l]
}
\]
 and similarly, for $\mfk{e}_{n,1}$
\[
\xymatrix{
0\ar[r]&0\ar[r]&0\ar[d]&&0\ar[r]&0\ar[r]&\Z_{(n,m)}\ar[d]\\
\Z_n\ar[u]&\Z\ar[l]&\Z\ar[l]&&\Z_{n/(n,m)}\ar[u]&\Z_{n}\ar[l]&\Z_n\ar[l]
}
\]

A map $\ksix(\mfk{e}_{n,1})\oplus\ksix(\mfk{e}_{n,1};\Z_n)\to
\ksix(\mfk{e}_{n,0})\oplus\ksix(\mfk{e}_{n,0};\Z_n)$
is given by a 12-tuple 
\[
((0,0,0,x,a,0),(0,0,b,c,d,0))
\]
where $x\in\Z$ and $a,b,c,d\in\Z_n$. To commute with the maps in the
diagrams as well as the Bockstein maps of type $\rho$ and $\beta$, we
must have $d=a$ and $c=\overline{x}$, and straightforward computations
show that this tuple extends uniquely to an element of 
$\mrm{Hom}_{ \Lambda } ( \tksix ( \mfk{e}_{n,1} ) , \tksix ( \mfk{e}_{
n , 0}  ) )$. Hence this group  is isomorphic to $\Z \oplus \Z_{n}
\oplus \Z_{n}$. Finally, note that no surjection $\Z\oplus \Z_n\to\Z\oplus \Z_n\oplus \Z_n$ exists.
\end{example}

\begin{remark}
The matrices
  $$A=\begin{smallpmatrix}
    1&1&0&0&0&0\\
    1&1&1&0&0&0\\
    0&1&1&0&0&0\\
    0&0&0&1&1&1\\
    0&0&0&1&1&1\\
    1&0&0&1&1&1
  \end{smallpmatrix}
  \quad\text{and}\quad 
  B=\begin{smallpmatrix}
    1&1&1&0&0&0\\
    1&1&1&0&0&0\\
    1&1&1&0&0&0\\
    0&0&0&1&1&0\\
    0&0&0&1&1&1\\
    1&0&0&0&1&1
  \end{smallpmatrix}$$
satisfy condition (II) of Cuntz
  (\cite{jc:cctmc2}). 
  Hence, the Cuntz-Krieger algebras $\mc{O}_A$ and $\mc{O}_B$ are
  purely infinite 
  \cstar-algebras and have exactly one non-trivial ideal. 
  Using the Smith normal form and \cite[Proposition 3.4]{gr:cckasi} we see that
  the six term exact sequence corresponding to $\mc{O}_A$ and
  $\mc{O}_B$ are (isomorphic to) the sequences 
  $$\vcenter{\xymatrix{
      \Z\ar[r]&\Z\ar[r]&\Z_2\ar[d]^0\\
      0\ar[u]&\Z\ar[l]&\Z\ar[l]^\cong
      }}
  \quad\text{and}\quad
  \vcenter{\xymatrix{\Z_2\ar[r]^0&\Z\ar[r]^\cong&\Z\ar[d]\\
      \Z\ar[u]&\Z\ar[l]&0\ar[l]}},$$
respectively.  Using $\kkE$-equivalent extensions, that $\kkE$ is split exact, and arguments similar to
  Example~\ref{ex:nonsplitting}, one easily shows that the natural map
  $\Gamma_{e_1,e_2}$ is not injective for the extensions $e_1$ and
  $e_2$ corresponding to the Cuntz-Krieger algebras $\mc{O}_A$ and
  $\mc{O}_B$, respectively. Similar considerations on 
  $$C=\begin{smallpmatrix}
    1&1&0&0&0&0\\
    1&1&1&0&0&0\\
    0&1&1&0&0&0\\
    0&0&0&1&1&1\\
    0&0&0&1&1&1\\
    1&0&0&1&1&1
  \end{smallpmatrix}
  \quad\text{and}\quad 
  D=\begin{smallpmatrix}
    1&1&0&0&0&0\\
    1&1&1&0&0&0\\
    0&1&1&0&0&0\\
    0&0&0&1&1&0\\
    0&0&0&1&1&1\\
    1&0&0&0&1&1
  \end{smallpmatrix}$$
yield a version of 
  Example~\ref{ex:nonsurjtksix} in the realm of Cuntz-Krieger algebras.
\end{remark}

One may ask if $\Gamma_{ e_{1} , e_{2}}$ is ever surjective and the answer is yes.  If $e_{1}$ is an extension of separable, nuclear \cstar-algebra in $\mc{N}$ such that the $K$-groups of $\ksix ( e_{1} )$ are torsion free, then $\mrm{Hom}_{ \Lambda } ( \tksix ( e_{1} ) , \tksix ( e_{2} ) )$ is naturally isomorphic to $\mrm{Hom}_{ \mc{Z}_{6} } ( \ksix ( e_{1} ) , \ksix ( e_{2} ) )$ such that the composition of $\Gamma_{ e_{1} , e_{2} }$ with this natural isomorphism is the natural map from $\kkE ( e_{1} , e_{2} )$ to $\mrm{Hom}_{ \mc{Z}_{6} } ( \ksix ( e_{1} ) , \ksix ( e_{2} ) )$.  Hence, by the UCT of Bonkat, we have that $\Gamma_{e_{1} , e_{2}}$ is surjective.

\section{Automorphisms of extensions of Kirchberg algebras}\label{sec-automorphisms}

The class $\ror$ of \cstar-algebras considered by R\o rdam in
\cite{mr:ceccstesk} consists of all $C^*$-algebras $A_1$ fitting in an essential
extension $e\colon A_0\hookrightarrow A_1\twoheadrightarrow A_2$ 
where $A_0$ and $A_2$ are Kirchberg algebras 
in $\mc{N}$ (with $A_0$ necessarily being stable).  For convenience we
shall often identify  $e$ and $A_1$ in this setting, as indeed
we can without risk of confusion. As explained in  \cite{mr:ceccstesk} one needs to consider three distinct cases: (1) $A_1$ is stable; (2)
$A_1$ is unital; and (3) $A_1$ is neither stable nor unital.

A functor $F$ is called a \emph{classification functor}, if
$A\cong B\Leftrightarrow F(A)\cong F(B)$ (for all algebras $A$ and
$B$ in the class considered). Such a functor $F$ is called a \emph{strong
classification functor} if every isomorphism from $F(A)$ to $F(B)$ is
induced by an isomorphism from $A$ to $B$ (for all algebras $A$ and $B$ in the
class considered).

R{\o}rdam in \cite{mr:ceccstesk} showed $\ksix$ to be a classification functor for
stable algebras in $\ror$. More recently, the authors in \cite{segr:rccconi} and \cite{grer:rccconiII}
showed that $\ksix$ (respectively $\ksix$ together with the class of the unit) is a
strong classification functor for stable (respectively unital) 
algebras in $\ror$. Moreover, they also showed that $\ksix$ is a classification
functor for non-stable, non-unital algebras in $\ror$.

In this section we will address some questions regarding the
automorphism group of $e$, where $e$ is in $\ror$.  If $e : A_{0} \hookrightarrow A_{1} \twoheadrightarrow A_{2}$ is an essential extension of separable \cstar-algebas,  then an automorphism of $e$ is a triple $( \phi_{0} , \phi_{1} , \phi_{2} )$ such that $\phi_{i}$ is an automorphism of $A_{i}$ and the diagram
\begin{equation*}
\xymatrix{
A_{0} \ar@{^{(}->}[r] \ar[d]^{ \phi_{0} } & A_{1} \ar@{>>}[r] \ar[d]^{ \phi_{1} } & A_{2} \ar[d]^{ \phi_{2} } \\
A_{0}  \ar@{^{(}->}[r]  & A_{1} \ar@{>>}[r]  & A_{2}
}
\end{equation*}
is commutative.  We denote the group of automorphisms of $e$ by
$\mrm{Aut}(e)$.  If $A_{0}$ and $A_{2}$ are simple \cstar-algebras,
then  $\mrm{Aut}(e)$ and $\mrm{Aut}(A_{1})$ are canonically
isomorphic.  Two automorphisms $(\phi_{0} , \phi_{1} , \phi_{2} )$, 
$(\psi_{0} , \psi_{1} , \psi_{2} )$ of $e$ are said to be asymptotically
(approximately) unitarily equivalent if $\phi_{1}$ and $\psi_{1}$ are
asymptotically (approximately) unitarily equivalent.  A consequence of
Kirchberg's results \cite{ek:nkmkna} is that $\kk_{ \mc{E} } ( e , e)$
classifies automorphisms of stable algebras in $\ror$.

In \cite{segr:rccconi} the first and second named authors asked
whether the canonical map from $\mrm{Aut}(e)$ to $\mrm{Aut}_\Lambda (
\tksix ( e ) )$ was surjective, cf.\ \eqref{autoeq1}. We answer this
in the negative as follows:

\begin{proposition}\label{prop:nonexact}
There is a \cstar-algebra $e\in\ror$ with finitely generated $K$-theory
such that (\ref{autoeq1}) is exact only at  
\begin{equation*}
\xymatrix{
\{ 1 \} \ar[r] & \overline{\mrm{Inn}} ( e ) \ar[r] & \mrm{Aut}(e)
}
\end{equation*}
\end{proposition}    

Before proving the above proposition we first need to set up some notation.  For $\phi$ in $\mrm{Aut} ( e )$, the element in $\kk_{ \mc{E} } ( e, e )$ induced by $\phi$ will be denoted by $\kk_{ \mc{E} } ( \phi )$ and the element in $\mrm{Hom}_{ \Lambda } ( \tksix ( e ) , \tksix (e) )$ induced by $\phi$ will be denoted by $\tksix ( \phi )$.  We will also need the following result.

\begin{proposition}\label{prop:isosp}
Let $e$ be any extension of separable \cstar-algebras.  Define
\begin{equation*}
\ftn{ \Lambda_{ \mfk{e}_{n,i} , e } }{ \kk_{ \mc{E} } ( \mfk{e}_{n ,i} , e ) }{ \mrm{Hom}_{ \Z } ( \kk_{ \mc{E} } ( \mfk{e}_{n,i} , \mfk{e}_{n,i} ) , \kk_{ \mc{E} }( \mfk{e}_{n,i} , e ) ) }
\end{equation*}
by $\Lambda_{ \mfk{e}_{ n , i } } ( x ) (y) = y \times x$, where $y \times x$ is the generalized Kasparov product (see \cite{ab:bkkpsc}).  Then $\Lambda_{ \mfk{e}_{n,i} , e }$ is an isomorphism for $i=0,1,2$. 
\end{proposition}

\begin{proof}
We will only prove the case when $i=0$, the other cases are similar.
By the UCT of Bonkat one shows that 
$\kk_{ \mc{E} } ( \mfk{e}_{n,0} , \mfk{e}_{n,0} )$ is isomorphic to 
$\Z$ and is generated by 
$\kk_{ \mc{E} } ( \id_{ \mfk{e}_{n,0} } )$.  
Therefore, if $\Lambda_{\mfk{e}_{n,0} , e } ( x ) = 0$, 
then 
$$x = \kk_{ \mc{E} } ( \id_{ \mfk{e}_{n,0} } ) \times x 
= \Lambda_{ \mfk{e}_{n,0} ,e } ( x ) 
(\kk_{ \mc{E} } ( \id_{ \mfk{e}_{n,0} } ) ) = 0.$$
Hence, 
$\Lambda_{ \mfk{e}_{n,0} }$ is injective.  Suppose $\alpha$ is a
homomorphism from	$\kk_{ \mc{E} } ( \mfk{e}_{n,0} ,
\mfk{e}_{n,0} )$  to 
$\kk_{ \mc{E} }( \mfk{e}_{n,0} , e ) $.  
Set $x = \alpha ( \kk_{ \mc{E} }( \id_{ \mfk{e}_{ n , 0 } } ) )$.  
Then 
$$\Lambda_{\mfk{e}_{n,0} , e } ( x ) 
( \kk_{ \mc{E} } ( \id_{ \mfk{e}_{n,0} } )) 
= x = \alpha ( \kk_{ \mc{E} } ( \id_{ \mfk{e}_{ n , 0 } } ) ).$$
Therefore, $ \Lambda_{ \mfk{e}_{n,0} , e }$ is surjective. 
\end{proof}
     
\noindent\emph{Proof of Proposition \ref{prop:nonexact}:}

Set $e_{1} = S \mfk{e}_{ p , 1} \oplus \mfk{e}_{ p , 1} \oplus
\mfk{e}_{p , 0}$ where $p$ is a prime number.  Let $\iota_{1}$ be the
embedding of $S \mfk{e}_{p,1}$ to $e_{1}$ and $\pi_{1}$ be the
projection from $e_{1}$ to $\mfk{e}_{p,0}$.  
Note that 
$$\ftn{ \kk_{ \mc{E} } ( \iota_{1} ) \times (-) }%
{ \kk_{ \mc{E} } ( \mfk{e}_{ p , 0 } , S \mfk{e}_{p,1} ) }%
{ \kk_{ \mc{E} } ( \mfk{e}_{p,0} ,  e_{1} ) }$$ 
and 
$$\ftn{ (-) \times \kk_{ \mc{E} } ( \pi_{1} ) }%
{ \kk_{ \mc{E} } ( \mfk{e}_{p,0} , e_{1} ) }%
{ \kk_{ \mc{E} } ( e_{1} , e_{1} ) }$$
are injective homomorphisms.  Hence 
$$\eta_{1} = ( (-)\times \kk_{ \mc{E} } ( \pi_{1} ) ) \circ 
( \kk_{ \mc{E} } ( \iota_{1}) \times (-) )$$
is injective.  Since $\Gamma_{ - , - }$ is natural 
\begin{equation*}
\xymatrix{
\kk_{ \mc{E} } ( \mfk{e}_{p,0} , S\mfk{e}_{p,1} ) \ar[d]_{ \Gamma_{ \mfk{e}_{p,0} , S \mfk{e}_{ p , 1 } } } \ar[r]^{ \eta_{1} } & \kk_{ \mc{E} } ( e_{1} , e_{1} ) \ar[d]_{ \Gamma_{ e_{1} , e_{1} } } \\
\mrm{Hom}_{ \Lambda } ( \tksix ( \mfk{e}_{p,0} ) , \tksix ( S \mfk{e}_{p,1} ) ) \ar[r]_{ \theta_{1} } & \mrm{Hom}_{ \Lambda } ( \tksix(e_{1}) , \tksix ( e_{1} ) )
}
\end{equation*}
is commutative.  By Example \ref{ex:nonsplitting}, $\Gamma_{ \mfk{e}_{p,0} , S \mfk{e}_{ p , 1 } }$ is not injective.  Therefore, $\Gamma_{ e_{1} , e_{1} }$ is not injective.  

Let $\pi_{2}$ be the projection of $e_{1}$ to $\mfk{e}_{p,0}$ and let
$\iota_{2}$ be the embedding of $\mfk{e}_{p,1}$ to $\mfk{e}_{1}$.
Note that 
$$\ftn{ \kk_{ \mc{E} } ( \pi_{2} ) \times (-)  }%
{ \kk_{ \mc{E} } ( e_{1} , e_{1} ) }%
{ \kk_{ \mc{E} } ( e_{1} , \mfk{e}_{p,0} ) }$$ 
and 
$$\ftn{ (-) \times \kk_{ \mc{E} } ( \iota_{2} ) }%
{ \kk_{ \mc{E} } ( e_{1} , \mfk{e}_{p,0} ) }%
{ \kk_{ \mc{E} } ( \mfk{e}_{p,1} , \mfk{e}_{p,0} ) }$$
are surjective homomorphisms.
Therefore, 
$$\eta_{2} = ( (-) \times \kk_{ \mc{E} } ( \iota_{2} ) )
\circ ( \kk_{ \mc{E} } ( \pi_{2} ) \times (-) )$$
is surjective.
Similarly, $\theta_{2} = \tksix ( \iota_{2} ) \circ \tksix ( \pi_{2}
)$ is surjective.  Since $\Gamma_{ - , - }$ is natural,  
\begin{equation*}
\xymatrix{
\kk_{ \mc{E} } ( e_{1} , e_{1} ) \ar[r]^{ \eta_{2} } \ar[d]_{ \Gamma_{ e_{1} , e_{1} } } & \kk_{ \mc{E} } ( \mfk{e}_{ p ,1 } , \mfk{e}_{ p , 0 } ) \ar[d]_{ \Gamma_{ \mfk{e}_{ p ,1 } , \mfk{e}_{p,0} } }  \\
\mrm{Hom}_{ \Lambda  } ( \tksix ( e_{1} ) , \tksix ( e_{1} ) ) \ar[r]_{ \theta_{2} } & \mrm{Hom}_{ \Lambda } ( \tksix ( \mfk{e}_{p,1} ) , \tksix ( \mfk{e}_{p,0} ) )
} 
\end{equation*}
is commutative.  By Example \ref{ex:nonsurjtksix}, $\Gamma_{ \mfk{e}_{p,1} , \mfk{e}_{p,0} }$ is not surjective.  Hence, $\Gamma_{e_{1} , e_{1}}$ is not surjective.

We have just shown that $\Gamma_{ e_{1} , e_{1} }$ is neither
surjective nor injective.  By Proposition 5.4 of \cite{mr:ceccstesk}
there is a stable extension $e : A_{0} \hookrightarrow A_{1}
\twoheadrightarrow A_{2}$ in $\ror$ such that $\ksix ( e ) \cong \ksix ( e_{1}
)$.  By the UCT of Bonkat, Theorem~\ref{thm:bonkatuct}, we are able to
lift this isomorphism to a $\kk_{ \mc{E} }$-equivalence.   
Therefore, 
\begin{equation*}
\xymatrix{
\kk_{ \mc{E} } ( e, e ) \ar[r]^{ \cong } \ar[d]^{ \Gamma_{e,e} } & \kk_{ \mc{E} } ( e_{1} , e_{1} ) \ar[d]^{ \Gamma_{e_{1} , e_{1} } }  \\
\mrm{Hom}_{ \Lambda } ( \tksix ( e ) , \tksix ( e ) ) \ar[r]_{ \cong } & \mrm{Hom}_{ \Lambda } ( \tksix ( e_{1} ) , \tksix ( e_{1} ) )
}
\end{equation*} 
is commutative.  Hence, $\Gamma_{e, e}$ is neither injective nor surjective..

Denote the kernel of the surjective map from 
\begin{equation*}
\mrm{Hom}_{ \Lambda } ( \tksix ( e ) , \tksix ( e ) ) \ \mathrm{to} \ \mrm{Hom}_{ \mc{Z}_{6} } ( \ksix ( e ) , \ksix ( e ) )
\end{equation*}
by $\mrm{Ext}_{ \mrm{six} } ( \ksix ( e ) , \ksix ( S e ) )$.  Note that if $\alpha$ is an element of $\mrm{Hom}_{ \Lambda } ( \tksix ( e ) , \tksix ( e ) )$ such that $\alpha \vert_{ \ksix ( e ) }$ is an isomorphism, then $\alpha$ is an isomorphism.  Since $\Gamma_{e,e}$ is not surjective and 
{\scriptsize
\begin{equation}\label{autoeq2}
\vcenter{\xymatrix{
\mrm{Ext}_{ \mc{Z}_{6} } ( \ksix ( e ) , \ksix ( S e ) ) \ar@{^{(}->}[r] \ar[d]^{ \Gamma_{e,e} } & \kk_{ \mc{E} } ( e, e ) \ar@{>>}[r] \ar[d]^{ \Gamma_{e,e} } & \mrm{Hom}_{ \mc{Z}_{6} } ( \ksix ( e ) , \ksix ( e ) ) \ar@{=}[d] \\
\mrm{Ext}_{ \mrm{six} } ( \ksix ( e ) , \ksix ( S e ) ) \ar@{^{(}->}[r] & \mrm{Hom}_{ \Lambda } ( \tksix ( e ) , \tksix ( e ) ) \ar@{>>}[r] & \mrm{Hom}_{ \mc{Z}_{6} } ( \ksix ( e ) , \ksix ( e ) )
}}
\end{equation}
}

\noindent is commutative, there exists $\beta_{1}$ in 
$\mrm{Ext}_{ \mrm{six} } ( \ksix ( e ) , \ksix ( S e ) )$ 
which is not in the image of $\Gamma_{e,e}$.  
Since $( \tksix( \id_{e}) + \beta_{1} ) \vert_{ \ksix ( e ) } 
= \tksix ( \id_{e} ) \vert_{ \ksix (e) }$, 
we have that $\tksix ( \id_{e} )+ \beta_{1}$ is an
automorphism of $\tksix ( e )$.  
Since $\beta_{1}$ is not in the image of $\Gamma_{e,e}$, $\tksix( \id_{e} ) + \beta_{1}$ is not in the image of $\Gamma_{e,e}$.  Hence, $\tksix( \id_{e} ) + \beta_{1}$ is an automorphism of $\tksix ( e )$ which does not lift to an automorphism of $e$.  Consequently,
\begin{equation*}
\xymatrix{
\mrm{Aut}(e) \ar[r] & \mrm{Aut}_\Lambda ( \tksix ( e ) ) \ar[r] & \{ 1 \}
}
\end{equation*}
is not exact.  

Since the diagram in (\ref{autoeq2}) is commutative and $\Gamma_{e,e}$
is not injective, there exists a nonzero element $\beta_{2}$ of
$\mrm{Ext}_{ \mc{Z}_{6} } ( \ksix ( e ) , \ksix ( S e ) )$ 
such that $\Gamma_{ e, e } ( \beta_{2} ) = 0$.  
Therefore, $\beta_{2} + \kk_{ \mc{E} }( \id_{ e } ) $ is 
an invertible element in $\kk_{ \mc{E} } ( e, e )$ such that 
$\Gamma_{e,e } ( \beta_{2} ) + \tksix ( \id_{e} ) 
= \tksix ( \id_{e})$.  
By Folgerung 4.3 of \cite{ek:nkmkna}, 
$\beta_{2} + \kk_{ \mc{E} } ( \id_{e} )$ 
lifts to an automorphism $\phi$ of $e$.  
So $\tksix ( \phi ) = \tksix ( \id_{e}  )$ in 
$\mrm{Aut}_\Lambda( \tksix ( e ) )$.  

Set 
\begin{eqnarray*}
G &=& \mrm{Hom} ( \kk_{ \mc{E} } ( S \mfk{e}_{p,1} , e_{1} ) , \kk_{ \mc{E} } ( S \mfk{e}_{ p , 1 } , e_{1} ) )\\ &&\oplus \left( \bigoplus_{ i = 0 }^{ 2} \mrm{Hom} ( \kk_{ \mc{E} } ( \mfk{e}_{ p, i } , e_{1} ) , \kk_{ \mc{E} } ( \mfk{e}_{ p, i } , e_{1} )  ) \right) \\
H &=& \mrm{Hom} ( \kk_{ \mc{E} } ( S \mfk{e}_{p,1} , e ) , \kk_{ \mc{E} } ( S \mfk{e}_{ p , 1 } , e ) )\\&& \oplus \left( \bigoplus_{ i = 0 }^{ 2} \mrm{Hom} ( \kk_{ \mc{E} } ( \mfk{e}_{ p, i } , e ) , \kk_{ \mc{E} } ( \mfk{e}_{ p, i } , e )  ) \right)
\end{eqnarray*}

Since $e_{1}$ is equal to $S \mfk{e}_{ p , 1} \oplus \mfk{e}_{ p , 1} \oplus \mfk{e}_{p , 0}$, by Proposition \ref{prop:isosp} 
the map from $\kk_{ \mc{E} } ( e_{1} , e_{1} )$ to $G$ given by $x
\mapsto ( - ) \times x$ is an isomorphism.  Hence, the map from 
$\kk_{  \mc{E} } ( e , e )$ to $H$ given by $x \mapsto ( - ) \times x$
is an isomorphism.  A computation shows that if $\phi$ is in 
$\overline{ \mrm{Inn} } ( e )$, then $\phi$ induces the identity
element in $H$. 
Therefore, $\phi$ is not approximately inner.  We have just shown that

\begin{equation*}
\xymatrix{
\overline{\mrm{Inn}} ( e ) \ar[r] & \mrm{Aut}(e) \ar[r] & \mrm{Aut}_\Lambda ( \tksix ( e ) )
}
\end{equation*}   
is not exact.
\qed

\section*{Acknowledgements}

We wish to thank the Fields Institute for the excellent working
conditions the second and third named authors enjoyed there when the
project was initiated, and for the similarly excellent conditions we
all enjoyed during the fall of 2007, where this work
was completed. We are also grateful to Marius Dadarlat for several
helpful discussions. Furthermore, the first named author wishes to
thank The Danish Foundation for Research in Natural Sciences for
financial support, and the second named author gratefully acknowledges the
support from ``Valdemar Andersens rejselegat for matematikere'' and
the Faroese Research Council. Finally, the third named author notes
his gratitude towards the hospitality enjoyed at the University of
Copenhagen at a visit during the spring of 2007.

\renewcommand{\sc}{}

\end{document}